\documentclass[12pt,english]{amsart}
\usepackage[T1]{fontenc}
\usepackage[latin9]{inputenc}
\usepackage{geometry}
\usepackage{babel}
\usepackage{amsthm}
\usepackage{amssymb}
\usepackage[unicode=true,pdfusetitle,
 bookmarks=true,bookmarksnumbered=false,bookmarksopen=false,
 breaklinks=false,pdfborder={0 0 1},backref=false,colorlinks=false]
 {hyperref}
\usepackage{breakurl}

\makeatletter
\theoremstyle{plain}
\newtheorem{thm}{\protect\theoremname}
  \theoremstyle{plain}
  \newtheorem{conjecture}[thm]{\protect\conjecturename}
  \theoremstyle{plain}
  \newtheorem{lem}[thm]{\protect\lemmaname}

\makeatother

  \providecommand{\conjecturename}{Conjecture}
  \providecommand{\lemmaname}{Lemma}
\providecommand{\theoremname}{Theorem}

\begin{document}

\title{Explicit formula for the average of Goldbach and prime tuples representations}

\author{Marco Cantarini}
\begin{abstract}
Let $\Lambda\left(n\right)$ be the Von Mangoldt function, let 
\[
r_{G}\left(n\right)=\underset{{\scriptstyle m_{1}+m_{2}=n}}{\sum_{m_{1},m_{2}\leq n}}\Lambda\left(m_{1}\right)\Lambda\left(m_{2}\right),
\]
\[
r_{PT}\left(N,h\right)=\sum_{n=0}^{N}\Lambda\left(n\right)\Lambda\left(n+h\right),\,h\in\mathbb{N}
\]
be the counting function of the Goldbach numbers and the counting
function of the prime tuples, respectively. Let $N>2$ be an integer.
We will find the explicit formulae for the averages of $r_{G}\left(n\right)$
and $r_{PT}\left(N,h\right)$ in terms of elementary functions, the
incomplete Beta function $B_{z}\left(a,b\right)$, series over $\rho$
that, with or without subscript, runs over the non-trivial zeros of
the Riemann Zeta function and the Dilogarithm function. We will also
prove the explicit formulae in an asymptotic form and a truncated
formula for the average of $r_{G}\left(n\right)$. Some observation
about these formulae and the average with Ces\`aro weight
\[
\frac{1}{\Gamma\left(k+1\right)}\sum_{n\leq N}r_{G}\left(n\right)\left(N-n\right)^{k},\,k>0
\]
are included.
\end{abstract}

\maketitle
\subjclass 2010 Mathematics Subject Classification:{ 11P32, 11N05}


\section{Introduction}

In this paper we prove an explicit formula and an asymptotic formula
for the average of the functions $r_{G}\left(n\right)$ and $r_{PT}\left(N,h\right)$,
which are the counting function of the Goldbach numbers and the counting
function of the prime tuples, respectively. This type of research
is classical; the first result for the average of counting function
of the Goldbach numbers was proved in 1991 by Fujii in a series of
paper \cite{Fuj1,Fuj2,Fuj3} writing two terms of the asymptotic expansion
with an error of $O\left(\left(N\log\left(N\right)\right)^{4/3}\right)$.
Then Granville \cite{Gran1,Gran2} gave the same result with a different
technique and Bhowmik and Schlage-Puchta \cite{Bho} improved the
error term to $O\left(N\log^{5}\left(N\right)\right)$. Finally, Languasco
and Zaccagnini \cite{LangZac1} were able to reach the error term
to $O\left(N\log^{3}\left(N\right)\right)$. In recent years there
has been some papers analyzing the weighed average with Ces\`aro weight
\begin{equation}
\frac{1}{\Gamma\left(k+1\right)}\sum_{n\leq N}r_{G}\left(n\right)\left(N-n\right)^{k},\,k>0,\label{eq:forma pesata}
\end{equation}
see \cite{Canta}, \cite{GoldYan} and \cite{LangZac2}. Even if the
technique of Languasco and Zaccagnini, developed to study \eqref{eq:forma pesata},
can be applied to various problems (see \cite{Canta2,Canta3,LangZac3})
in all of these papers there are some limitations over the parameter
$k$ due to some convergence problems. In a very recent paper Br\"udern,
Kaczorowski and Perelli \cite{Perelli} were able to find an explicit
formula which holds for all $k>0.$ We present an approach that analyzes
the pure average form or, in other words, the case $k=0.$ We will
find an explicit formula and we will prove that it is possible write
it as an asymptotic formula with three terms and an error term $O\left(N\right)$
without the assumption of the Riemann hypothesis (RH for brevity).
We will prove the following
\begin{thm}
Let $N>2$ be an integer. Then 
\begin{align*}
\sideset{}{^{\prime}}\sum_{n\leq2N}r_{G}\left(n\right)= & 2N^{2}-2\sum_{\rho}\frac{\left(2N-2\right)^{\rho+1}}{\rho\left(\rho+1\right)}\\
+ & 2\sum_{\rho_{1}}\left(2N\right)^{\rho_{1}}\left(\Gamma\left(\rho_{1}\right)\sum_{\rho_{2}}\frac{\left(2N\right)^{\rho_{2}}\Gamma\left(\rho_{2}\right)}{\Gamma\left(\rho_{1}+\rho_{2}+1\right)}\right.\\
- & \left.\sum_{\rho_{2}}\frac{\left(2N\right)^{\rho_{2}}}{\rho_{2}}\left(B_{1/N}\left(\rho_{2}+1,\rho_{1}\right)+B_{1/2}\left(\rho_{1},\rho_{2}+1\right)\right)\right)\\
+ & F\left(N\right)
\end{align*}
where
\[
\sideset{}{^{\prime}}\sum_{n\leq2N}r_{G}\left(n\right)=\sum_{n\leq2N}r_{G}\left(n\right)-\frac{r_{G}\left(2N\right)}{2}
\]
$B_{z}\left(a,b\right)$ is the incomplete Beta function, $\rho=\beta+i\gamma$,
with or without subscript, runs over the non-trivial zeros of the
Riemann Zeta function $\zeta\left(s\right)$ and $F\left(N\right)$
is a function that can be explicitly calculated in terms of elementary
functions, series over non-trivial zeros, Dilogarithm and incomplete
Beta functions and with 
\[
F\left(N\right)=O\left(N\right)
\]
as $N\rightarrow\infty.$ Furthermore for all $T^{\prime},\,T^{\prime\prime}>2$
we have
\begin{align*}
\sideset{}{^{\prime}}\sum_{n\leq2N}r_{G}\left(n\right)= & 2N^{2}-2\sum_{\rho:\,\left|\gamma\right|\leq T^{\prime}}\frac{\left(2N-2\right)^{\rho+1}}{\rho\left(\rho+1\right)}\\
+ & 2\sum_{\rho_{1}:\,\left|\gamma_{1}\right|\leq T^{\prime}}\left(2N\right)^{\rho_{1}}\left(\Gamma\left(\rho_{1}\right)\sum_{\rho_{2}:\,\left|\gamma_{2}\right|\leq T^{\prime\prime}}\frac{\left(2N\right)^{\rho_{2}}\Gamma\left(\rho_{2}\right)}{\Gamma\left(\rho_{1}+\rho_{2}+1\right)}\right.\\
- & \left.\sum_{\rho_{2}:\,\left|\gamma_{2}\right|\leq T^{\prime\prime}}\frac{\left(2N\right)^{\rho_{2}}}{\rho_{2}}\left(B_{1/N}\left(\rho_{2}+1,\rho_{1}\right)+B_{1/2}\left(\rho_{1},\rho_{2}+1\right)\right)\right)\\
+ & 2\sum_{\rho_{1}:\,\left|\gamma_{1}\right|\leq T^{\prime}}\frac{N^{\rho_{1}}}{\rho_{1}}\sum_{\rho_{2}:\,\left|\gamma_{2}\right|\leq T^{\prime\prime}}\frac{N^{\rho_{2}}}{\rho_{2}}-2\left(\sum_{\rho_{1}:\,\left|\gamma_{1}\right|\leq T^{\prime}}\frac{N^{\rho_{1}}}{\rho_{1}}\right)^{2}\\
+ & F\left(N,T^{\prime},T^{\prime\prime}\right)+O\left(\frac{N\log^{2}\left(NT^{\prime\prime}\right)T^{\prime}\log\left(T^{\prime}\right)G\left(N\right)}{T^{\prime\prime}}+\frac{N^{2}\log^{2}\left(T^{\prime}N\right)}{T^{\prime}}\right)
\end{align*}
where 
\[
G\left(N\right)=\begin{cases}
N\exp\left(-C\sqrt{\log\left(N\right)}\right) & \mathrm{without}\,\mathrm{RH}\\
\sqrt{N}\log^{2}\left(N\right) & \mathrm{with}\,\mathrm{RH}
\end{cases}
\]
and $F\left(N,T^{\prime},T^{\prime\prime}\right)$ is a function that
can be explicitly calculated in terms of elementary functions, series
over non-trivial zeros, Dilogarithm and incomplete Beta functions
and with the property
\[
F\left(N,T^{\prime},T^{\prime\prime}\right)\ll N
\]
where the implicit constant does not depend on $T^{\prime}$ and $T^{\prime\prime}$.
\end{thm}
Note that the term 
\[
\sum_{\rho_{1}}\left(2N\right)^{\rho_{1}}\Gamma\left(\rho_{1}\right)\sum_{\rho_{2}}\frac{\left(2N\right)^{\rho_{2}}\Gamma\left(\rho_{2}\right)}{\Gamma\left(\rho_{1}+\rho_{2}+1\right)}
\]
is what we expect considering the formula in \cite{Perelli} and taking
$k=0.$ It is interesting to note that if we assume the third term
of the explicit formula in Theorem 1 grows in a suitable way as $N\rightarrow\infty$
then we can prove that every interval $\left[2N,2N+2H\right]$, where
$H=H\left(N\right)$ is a function of $N$ that grows in a suitable
way, contains a Goldbach number. More precisely, we propose the following
conjecture
\begin{conjecture}
\label{congettura}Under RH we have the estimation
\[
\sum_{\rho_{1}}\left(2N\right)^{\rho_{1}}\left(\Gamma\left(\rho_{1}\right)\sum_{\rho_{2}}\frac{\left(2N\right)^{\rho_{2}}\Gamma\left(\rho_{2}\right)}{\Gamma\left(\rho_{1}+\rho_{2}+1\right)}-\sum_{\rho_{2}}\frac{\left(2N\right)^{\rho_{2}}}{\rho_{2}}\left(B_{1/N}\left(\rho_{2}+1,\rho_{1}\right)+B_{1/2}\left(\rho_{1},\rho_{2}+1\right)\right)\right)\ll N.
\]
\end{conjecture}
If Conjecture \ref{congettura} holds then we get the following
\begin{thm}
Assume that Conjecture \ref{congettura} holds. Then in every interval
$\left[2N,2N+2H\right]$ where $H=H\left(N\right)$ is 
\begin{equation}
H\left(N\right)=C\,\log\left(\log\left(N\right)\right)\label{eq:crescita H}
\end{equation}
where $C>0$ is a sufficiently large constant contains a Goldbach
number.
\end{thm}
\begin{proof}
From Theorem 1 and Conjecture \ref{congettura} we get, for every
$H>0$, that
\begin{align*}
\sum_{n=2N+1}^{2N+2H}r_{G}\left(n\right)= & 4NH+2H^{2}-2\sum_{\rho}\frac{\left(2N+2H-2\right)^{\rho+1}}{\rho\left(\rho+1\right)}+2\sum_{\rho}\frac{\left(2N-2\right)^{\rho+1}}{\rho\left(\rho+1\right)}\\
+ & \frac{r_{G}\left(2N+2H\right)}{2}-\frac{r_{G}\left(2N\right)}{2}+O\left(N\right)
\end{align*}
then we can observe that
\[
-2\sum_{\rho}\frac{\left(2N+2H-2\right)^{\rho+1}}{\rho\left(\rho+1\right)}+2\sum_{\rho}\frac{\left(2N-2\right)^{\rho+1}}{\rho\left(\rho+1\right)}=-2\int_{2N-2}^{2N+2H-2}\sum_{\rho}\frac{t^{\rho}}{\rho}dt
\]
since we know we can switch the integral and the series over the non-trivial
zeros (see Lemma \ref{lem:scambio formula expl}). So, by \eqref{eq:stima zeri classica},
we get
\[
-2\sum_{\rho}\frac{\left(2N+2H-2\right)^{\rho+1}}{\rho\left(\rho+1\right)}+2\sum_{\rho}\frac{\left(2N-2\right)^{\rho+1}}{\rho\left(\rho+1\right)}\ll H\sqrt{N}\log^{2}\left(2N+2H\right)
\]
and since
\[
r_{G}\left(M\right)\ll M\mathfrak{S}\left(M\right)\ll M\,\log\left(\log\left(M\right)\right),
\]
where
\[
\mathfrak{S}\left(M\right)=2\prod_{p>2}\left(1-\frac{1}{\left(p-2\right)^{2}}\right)\prod_{p\mid M:\,p>2}\frac{p-1}{p-2},\,p\,\mathrm{prime}\,\mathrm{number},
\]
if $M$ is even and vanishes if $M$ is odd (see for example \cite{MoVa},
Theorem $3.13$), then if we take $H=H\left(N\right)$ as in \eqref{eq:crescita H}
we get the thesis. 
\end{proof}
Probably, with a more accurate analysis, it is possible to obtain
$H\left(N\right)$ as a large constant but it is not the aim of this
paper. 

We need some comments about the truncated formula. This form is interesting
since allows to work with finite sums instead of series and so it
is reasonable to think that, with a clever choice of $T^{\prime}$
and $T^{\prime\prime}$, we can estimate the double sums efficiently.
As we will see in the proof the error term in the formula strictly
depends on the choice of some parameters and, probably, is not optimized;
we expect that a better analysis can be done and this will be the
subject of future research.

In Section 4 we will also talk about the possibility to use our method
to calculate the weighed form \eqref{eq:forma pesata}. 

Let us talk about the average of prime tuples. We can recall Bombieri
and Davemport \cite{BomDav}, Maier and Pomerance \cite{MaPo} and
Balog \cite{Bal}, which obtained a ``Bombieri-Vinogradov type results''.
We will prove the following
\begin{thm}
Let $N>2$ and $0\leq M\leq N$ be integers. Then 
\begin{align}
\sideset{}{^{\prime}}\sum_{h=0}^{M}r_{PT}\left(N,h\right)= & NM+\sum_{\rho}\frac{N^{\rho+1}}{\rho\left(\rho+1\right)}-\sum_{\rho}\frac{\left(N+M\right)^{\rho+1}}{\rho\left(\rho+1\right)}+\sum_{\rho}\frac{\left(2+M\right)^{\rho+1}}{\rho\left(\rho+1\right)}\nonumber \\
- & \left(\sum_{\rho_{1}}\sum_{\rho_{2}}\frac{M^{\rho_{1}+\rho_{2}}\left(-1\right)^{\rho_{2}+1}}{\rho_{2}}\left(B_{-2/M}\left(\rho_{2}+1,\rho_{1}\right)-B_{-N/M}\left(\rho_{2}+1,\rho_{1}\right)\right)\right)\cdot1\left(M\right)\label{eq:gemelli 1}\\
- & \left(\sum_{\rho_{1}}\sum_{\rho_{2}}\frac{N^{\rho_{1}+\rho_{2}}-2^{\rho_{1}+\rho_{2}}}{\rho_{2}\left(\rho_{1}+\rho_{2}\right)}\right)\cdot\widetilde{1}\left(M\right)+\sum_{\rho_{1}}\frac{N^{\rho_{1}}}{\rho_{1}}\sum_{\rho_{2}}\frac{\left(N+M\right)^{\rho_{2}}}{\rho_{2}}\label{eq:gemelli 2}\\
+ & \frac{\Lambda\left(N\right)}{2}\sum_{\rho}\frac{N^{\rho}}{\rho}-\frac{\Lambda\left(N\right)}{2}\sum_{\rho}\frac{\left(N+M\right)^{\rho}}{\rho}\\
+ & G\left(N,M\right)\nonumber 
\end{align}
where 
\[
1\left(M\right)=\begin{cases}
0, & M=0\\
1, & M>0,
\end{cases}
\]
\[
\widetilde{1}\left(M\right)=\begin{cases}
2, & M=0\\
1, & M>0,
\end{cases}
\]
\[
\sideset{}{^{\prime}}\sum_{h=0}^{M}r_{PT}\left(N,h\right)=\sum_{h=0}^{M}r_{PT}\left(N,h\right)-\frac{r_{PT}\left(N,M\right)}{2}-\frac{r_{PT}\left(N,0\right)}{2},
\]
and $G\left(N,M\right)$ is a function that can be explicitly calculated
in terms of special functions like the incomplete Beta function and
with the property
\[
G\left(N,M\right)\ll\begin{cases}
N\left(M+1\right)\exp\left(-C\sqrt{\log\left(N\right)}\right), & \mathrm{without}\,\mathrm{RH}\\
\sqrt{N}\left(M+1\right)\log^{2}\left(N\right), & \mathrm{with}\,\mathrm{RH}
\end{cases}
\]
and $C>0$ is a real constant and the implicit constant does not depend
on $M$. 
\end{thm}
Again we can observe that a precise control of the series in \eqref{eq:gemelli 1}
and \eqref{eq:gemelli 2} allows us to obtain information on the sum
$r_{PT}\left(N,h\right)$ with a fixed $M.$ 

I thank my mentor Alessandro Zaccagnini for a discussion on this topic.

\section{Lemmas}

We recall a Lemma that we use several times. 
\begin{lem}
\label{lem:scambio formula expl}Let $g$ be a continuously differentiable
function on $\left[a,b\right]$ with $2\leq a\leq b<\infty$ and $\psi\left(t\right)$
the Chebyshev psi function. We have
\begin{align*}
\int_{a}^{b}\psi\left(t\right)g\left(t\right)dt= & \int_{a}^{b}tg\left(t\right)dt-\sum_{\rho}\frac{1}{\rho}\int_{a}^{b}t^{\rho}g\left(t\right)dt\\
- & \int_{a}^{b}\left(\frac{\zeta^{\prime}}{\zeta}\left(0\right)+\frac{\log\left(1-1/t^{2}\right)}{2}\right)g\left(t\right)dt.
\end{align*}
\end{lem}
The proof can be found in \cite{Ram}, Lemma 4. The formula can be
extended to $b=\infty$ assuming that $g\left(t\right)$ decays at
$+\infty$ sufficiently fast (an example is present in \cite{Canta}).
Also, from the proof of the Lemma, is clear that the the formula holds
even if $g\left(t\right)\in\mathbb{C}$ with the hypothesis
\[
\int_{a}^{b}\left|g\left(t\right)\right|dt<\infty.
\]

Now we present our fundamental lemma. 
\begin{lem}
\label{lemma fondamentale}Let $x,y\in\mathbb{R},\,3\leq x\leq y$
and $\alpha\in\mathbb{C},\,\mathrm{Re}\left(\alpha\right)>0.$ Then
\begin{align*}
\sum_{n\leq\left\lfloor x\right\rfloor }\Lambda\left(n\right)\left(1-\frac{n}{y}\right)^{\alpha}= & \frac{y}{\alpha+1}\left(\left(1-\frac{2}{y}\right)^{\alpha+1}-\left(1-\frac{\left\lfloor x\right\rfloor }{y}\right)^{\alpha+1}\right)+2\left(1-\frac{2}{y}\right)^{\alpha}\\
- & \sum_{\rho}y^{\rho}\frac{\Gamma\left(\rho\right)\Gamma\left(\alpha+1\right)}{\Gamma\left(\rho+1+\alpha\right)}+\alpha\sum_{\rho}\frac{y^{\rho}}{\rho}\left(B_{2/y}\left(\rho+1,\alpha\right)+B_{(y-\left\lfloor x\right\rfloor )/y}\left(\alpha,\rho+1\right)\right)\\
- & \left(1-\frac{\left\lfloor x\right\rfloor }{y}\right)^{\alpha}\sum_{\rho}\frac{\left\lfloor x\right\rfloor ^{\rho}}{\rho}-\frac{\zeta^{\prime}}{\zeta}\left(0\right)\left(1-\frac{2}{y}\right)^{\alpha}-\frac{1}{2}\log\left(\frac{3}{4}\right)\left(1-\frac{2}{y}\right)^{\alpha}\\
+ & \omega\left(\alpha,y,x\right)+\frac{\Lambda\left(\left\lfloor x\right\rfloor \right)}{2}\left(1-\frac{\left\lfloor x\right\rfloor }{y}\right)^{\alpha}
\end{align*}
where $\left\lfloor x\right\rfloor $ is the floor function, $\psi\left(s\right)$
is the Chebyshev psi function,
\begin{align*}
\omega\left(\alpha,y,x\right)= & -\frac{\left(y-1\right)^{\alpha}}{2y^{\alpha}}\left(B_{(y-2)/(y-1)}\left(\alpha+1,0\right)-B_{(y-\left\lfloor x\right\rfloor )/(y-1)}\left(\alpha+1,0\right)\right)\\
- & \frac{\left(y+1\right)^{\alpha}}{2y^{\alpha}}\left(B_{(y-2)/(y+1)}\left(\alpha+1,0\right)-B_{(y-\left\lfloor x\right\rfloor )/(y+1)}\left(\alpha+1,0\right)\right)\\
+ & B_{(y-2)/y}\left(\alpha+1,0\right)-B_{(y-\left\lfloor x\right\rfloor )/y}\left(\alpha+1,0\right)
\end{align*}
 and $B_{z}\left(a,b\right)$ is the incomplete Beta function, with
the convention $B_{0}\left(a,b\right)=0.$ Furthermore if $y>x$ then
for all $T>2$ we have
\begin{align*}
\sum_{n\leq\left\lfloor x\right\rfloor }\Lambda\left(n\right)\left(1-\frac{n}{y}\right)^{\alpha}= & \frac{y}{\alpha+1}\left(\left(1-\frac{2}{y}\right)^{\alpha+1}-\left(1-\frac{\left\lfloor x\right\rfloor }{y}\right)^{\alpha+1}\right)+2\left(1-\frac{2}{y}\right)^{\alpha}\\
- & \sum_{\rho:\,\left|\gamma\right|\leq T}y^{\rho}\frac{\Gamma\left(\rho\right)\Gamma\left(\alpha+1\right)}{\Gamma\left(\rho+1+\alpha\right)}+\alpha\sum_{\rho:\,\left|\gamma\right|\leq T}\frac{y^{\rho}}{\rho}\left(B_{2/y}\left(\rho+1,\alpha\right)+B_{(y-\left\lfloor x\right\rfloor )/y}\left(\alpha,\rho+1\right)\right)\\
- & \left(1-\frac{\left\lfloor x\right\rfloor }{y}\right)^{\alpha}\sum_{\rho:\,\left|\gamma\right|\leq T}\frac{\left\lfloor x\right\rfloor ^{\rho}}{\rho}-\frac{\zeta^{\prime}}{\zeta}\left(0\right)\left(1-\frac{2}{y}\right)^{\alpha}-\frac{1}{2}\log\left(\frac{3}{4}\right)\left(1-\frac{2}{y}\right)^{\alpha}\\
+ & \omega\left(\alpha,y,x\right)+\left(1-\frac{\left\lfloor x\right\rfloor }{y}\right)^{\alpha}\frac{\Lambda\left(\left\lfloor x\right\rfloor \right)}{2}\\
+ & O\left(M\left(\alpha,y,x\right)\frac{\left|\alpha\right|\left\lfloor x\right\rfloor ^{2}\log^{2}\left(\left\lfloor x\right\rfloor T\right)}{yT}+\frac{\left\lfloor x\right\rfloor \log^{2}\left(\left\lfloor x\right\rfloor T\right)}{T}\right)
\end{align*}
where the implicit constant in the error term does not depend on $\alpha,y$
and $x$ and
\[
M\left(\alpha,y,x\right)=\begin{cases}
1, & \mathrm{Re}\left(\alpha\right)\geq1\\
\left(1-\frac{\left\lfloor x\right\rfloor }{y}\right)^{\mathrm{Re}\left(\alpha\right)-1}, & 0<\mathrm{Re}\left(\alpha\right)<1.
\end{cases}
\]
\end{lem}
\begin{proof}
By the Abel summation formula we have
\begin{equation}
\sum_{n\leq\left\lfloor x\right\rfloor }\Lambda\left(n\right)\left(1-\frac{n}{y}\right)^{\alpha}=\psi\left(\left\lfloor x\right\rfloor \right)\left(1-\frac{\left\lfloor x\right\rfloor }{y}\right)^{\alpha}+\frac{\alpha}{y}\int_{2}^{\left\lfloor x\right\rfloor }\psi\left(t\right)\left(1-\frac{t}{y}\right)^{\alpha-1}dt\label{eq:lemma fundment riga 1}
\end{equation}
where $\left\lfloor x\right\rfloor $ is the floor function, so by
the explicit formula for $\psi_{0}\left(t\right)=\psi\left(t\right)-\frac{\Lambda\left(t\right)}{2}$
(or by Lemma \ref{lem:scambio formula expl}) we have
\begin{align}
\sum_{n\leq\left\lfloor x\right\rfloor }\Lambda\left(n\right)\left(1-\frac{n}{y}\right)^{\alpha}= & \psi\left(\left\lfloor x\right\rfloor \right)\left(1-\frac{\left\lfloor x\right\rfloor }{y}\right)^{\alpha}\nonumber \\
+ & \frac{\alpha}{y}\int_{2}^{\left\lfloor x\right\rfloor }\left(t-\sum_{\rho}\frac{t^{\rho}}{\rho}-\frac{\zeta^{\prime}}{\zeta}\left(0\right)-\frac{1}{2}\log\left(1-\frac{1}{t^{2}}\right)\right)\left(1-\frac{t}{y}\right)^{\alpha-1}dt\nonumber \\
= & \psi\left(\left\lfloor x\right\rfloor \right)\left(1-\frac{\left\lfloor x\right\rfloor }{y}\right)^{\alpha}+\frac{\alpha}{y}\sum_{w=1}^{4}\int_{2}^{\left\lfloor x\right\rfloor }g_{w}\left(t\right)\left(1-\frac{t}{y}\right)^{\alpha-1}dt\label{eq:psi floor}
\end{align}
where $g_{w}\left(t\right)$ are the terms of the explicit formula
of $\psi_{0}\left(t\right).$

$\ $

\textbf{Integral of $g_{1}\left(t\right)$}

We have to calculate
\[
\frac{\alpha}{y}\int_{2}^{\left\lfloor x\right\rfloor }g_{1}\left(t\right)\left(1-\frac{t}{y}\right)^{\alpha-1}dt=\frac{\alpha}{y}\int_{2}^{\left\lfloor x\right\rfloor }t\left(1-\frac{t}{y}\right)^{\alpha-1}dt
\]
then taking $t/y=u$ and integrating by parts we get
\[
\frac{\alpha}{y}\int_{2}^{\left\lfloor x\right\rfloor }g_{1}\left(t\right)\left(1-\frac{t}{y}\right)^{\alpha-1}dt=y\alpha\int_{2/y}^{\left\lfloor x\right\rfloor /y}u\left(1-u\right)^{\alpha-1}du
\]
\[
=\frac{y}{\alpha+1}\left(\left(1-\frac{2}{y}\right)^{\alpha+1}-\left(1-\frac{\left\lfloor x\right\rfloor }{y}\right)^{\alpha+1}\right)-\left\lfloor x\right\rfloor \left(1-\frac{\left\lfloor x\right\rfloor }{y}\right)^{\alpha}+2\left(1-\frac{2}{y}\right)^{\alpha}.
\]

$\ $

\textbf{Integral of $g_{2}\left(t\right)$}

We have to estimate
\[
\frac{\alpha}{y}\int_{2}^{\left\lfloor x\right\rfloor }g_{2}\left(t\right)\left(1-\frac{t}{y}\right)^{\alpha-1}dt=-\frac{\alpha}{y}\int_{2}^{\left\lfloor x\right\rfloor }\sum_{\rho}\frac{t^{\rho}}{\rho}\left(1-\frac{t}{y}\right)^{\alpha-1}dt
\]
then by Lemma \ref{lem:scambio formula expl} we we know that we can
exchange the integral with the series so
\[
\frac{\alpha}{y}\int_{2}^{\left\lfloor x\right\rfloor }g_{2}\left(t\right)\left(1-\frac{t}{y}\right)^{\alpha-1}dt=-\frac{\alpha}{y}\sum_{\rho}\frac{1}{\rho}\int_{2}^{\left\lfloor x\right\rfloor }t^{\rho}\left(1-\frac{t}{y}\right)^{\alpha-1}dt=-\alpha\sum_{\rho}\frac{y^{\rho}}{\rho}\int_{2/y}^{\left\lfloor x\right\rfloor /y}u^{\rho}\left(1-u\right)^{\alpha-1}du
\]
\[
=-\sum_{\rho}y^{\rho}\frac{\Gamma\left(\rho\right)\Gamma\left(\alpha+1\right)}{\Gamma\left(\rho+1+\alpha\right)}+\alpha\sum_{\rho}\frac{y^{\rho}}{\rho}\left(B_{2/y}\left(\rho+1,\alpha\right)+B_{(y-\left\lfloor x\right\rfloor )/y}\left(\alpha,\rho+1\right)\right)
\]
where $B_{z}\left(a,b\right)$ is the incomplete Beta function (for
details see for example \cite{Olver}, chapter $8.17$). Note that
the last identity is valid since $\sum_{\rho}y^{\rho}\frac{\Gamma\left(\rho\right)}{\Gamma\left(\rho+1+\alpha\right)}$
is absolutely and compactly convergent for $\mathrm{Re}\left(\alpha\right)>0$
(see \cite{Perelli}). 

$\ $

\textbf{Integral of $g_{3}\left(t\right)$}

Trivially we have
\[
\frac{\alpha}{y}\int_{2}^{\left\lfloor x\right\rfloor }g_{3}\left(t\right)\left(1-\frac{t}{y}\right)^{\alpha-1}dt=-\frac{\zeta^{\prime}}{\zeta}\left(0\right)\frac{\alpha}{y}\int_{2}^{\left\lfloor x\right\rfloor }\left(1-\frac{t}{y}\right)^{\alpha-1}dt=\frac{\zeta^{\prime}}{\zeta}\left(0\right)\left(1-\frac{\left\lfloor x\right\rfloor }{y}\right)^{\alpha}-\frac{\zeta^{\prime}}{\zeta}\left(0\right)\left(1-\frac{2}{y}\right)^{\alpha}.
\]

$\ $

\textbf{Integral of $g_{4}\left(t\right)$}

Integrating by parts we have that
\[
\frac{\alpha}{y}\int_{2}^{\left\lfloor x\right\rfloor }g_{4}\left(t\right)\left(1-\frac{t}{y}\right)^{\alpha-1}dt=-\frac{\alpha}{2y}\int_{2}^{\left\lfloor x\right\rfloor }\log\left(1-\frac{1}{t^{2}}\right)\left(1-\frac{t}{y}\right)^{\alpha-1}dt
\]
\[
=\frac{1}{2}\log\left(1-\frac{1}{\left\lfloor x\right\rfloor ^{2}}\right)\left(1-\frac{\left\lfloor x\right\rfloor }{y}\right)^{\alpha}-\frac{1}{2}\log\left(\frac{3}{4}\right)\left(1-\frac{2}{y}\right)^{\alpha}
\]
\begin{equation}
+\frac{1}{y^{\alpha}}\int_{2}^{\left\lfloor x\right\rfloor }t^{-1}\left(y-t\right)^{\alpha}dt-\frac{1}{2y^{\alpha}}\int_{2}^{\left\lfloor x\right\rfloor }\left(t+1\right)^{-1}\left(y-t\right)^{\alpha}dt-\frac{1}{2y^{\alpha}}\int_{2}^{\left\lfloor x\right\rfloor }\left(t-1\right)^{-1}\left(y-t\right)^{\alpha}dt.\label{eq:contacci}
\end{equation}
Now we calculate explicitly only the first integral of \eqref{eq:contacci}
since the other are similar. Taking $\frac{y-t}{y}=s$ we get
\[
\frac{1}{y^{\alpha}}\int_{2}^{\left\lfloor x\right\rfloor }t^{-1}\left(y-t\right)^{\alpha}dt=\int_{\left(y-\left\lfloor x\right\rfloor \right)/y}^{\left(y-2\right)/y}\left(1-s\right)^{-1}s^{\alpha}ds
\]
\[
=B_{\left(y-2\right)/y}\left(\alpha+1,0\right)-B_{\left(y-\left\lfloor x\right\rfloor \right)/y}\left(\alpha+1,0\right)
\]
then, arguing in this way for all the integrals in \eqref{eq:contacci}
and expanding $\psi\left(\left\lfloor x\right\rfloor \right)$ with
its explicit formula, we get the thesis.

For the proof of the truncated version we use the formula 
\begin{equation}
\psi_{0}\left(x\right)=x-\sum_{\rho:\,\left|\gamma\right|\leq T}\frac{x^{\rho}}{\rho}-\frac{\zeta^{\prime}}{\zeta}\left(0\right)-\frac{\log\left(1-1/x^{2}\right)}{2}+O\left(\frac{x\log^{2}\left(Tx\right)}{T}+\log\left(x\right)\min\left(1,\frac{x}{T\left\langle x\right\rangle }\right)\right),\,T>2,\,x>1\label{troncata}
\end{equation}
(see for example \cite{Daven}, formulae $(9)$ and $(10)$ at page
$109$) where $\left\langle x\right\rangle $ is the distance of $x$
to the nearest prime power other than $x$, then, substituting \eqref{troncata}
in \eqref{eq:lemma fundment riga 1}, we can see that the problem
boils down to evaluate
\begin{equation}
\frac{\left|\alpha\right|}{yT}\int_{2}^{\left\lfloor x\right\rfloor }t\log^{2}\left(tT\right)\left(1-\frac{t}{y}\right)^{\mathrm{Re}\left(\alpha\right)-1}dt\label{eq:int error 1}
\end{equation}
and
\begin{equation}
\frac{\left|\alpha\right|}{y}\int_{2}^{\left\lfloor x\right\rfloor }\log\left(t\right)\min\left(1,\frac{t}{T\left\langle t\right\rangle }\right)\left(1-\frac{t}{y}\right)^{\mathrm{Re}\left(\alpha\right)-1}dt.\label{eq:int error 2}
\end{equation}
Let us define
\[
M\left(\alpha,y,x\right)=\begin{cases}
1, & \mathrm{Re}\left(\alpha\right)\geq1\\
\left(1-\frac{\left\lfloor x\right\rfloor }{y}\right)^{\mathrm{Re}\left(\alpha\right)-1}, & 0<\mathrm{Re}\left(\alpha\right)<1.
\end{cases}
\]
Then
\[
\frac{\left|\alpha\right|}{yT}\int_{2}^{\left\lfloor x\right\rfloor }t\log^{2}\left(tT\right)\left(1-\frac{t}{y}\right)^{\mathrm{Re}\left(\alpha\right)-1}dt\ll\frac{\left|\alpha\right|}{yT}M\left(\alpha,y,x\right)\int_{2}^{\left\lfloor x\right\rfloor }t\log^{2}\left(tT\right)dt
\]
\begin{equation}
\ll M\left(\alpha,y,x\right)\frac{\left|\alpha\right|\left\lfloor x\right\rfloor ^{2}\log^{2}\left(\left\lfloor x\right\rfloor T\right)}{yT}.\label{eq:int error term 1}
\end{equation}
We now analyze \eqref{eq:int error 2}. Let us define
\[
\Omega_{1}=\left\{ t\in\left(2,\left\lfloor x\right\rfloor \right):\,\left\langle t\right\rangle \geq1\right\} 
\]
then we can easily see that
\[
\frac{\left|\alpha\right|}{y}\int_{\Omega_{1}}\log\left(t\right)\min\left(1,\frac{t}{T\left\langle t\right\rangle }\right)\left(1-\frac{t}{y}\right)^{\mathrm{Re}\left(\alpha\right)-1}dt\ll\frac{\left|\alpha\right|}{yT}M\left(\alpha,y,x\right)\int_{2}^{\left\lfloor x\right\rfloor }t\log\left(t\right)dt
\]
\[
\ll M\left(\alpha,y,x\right)\frac{\left|\alpha\right|\left\lfloor x\right\rfloor ^{2}\log\left(\left\lfloor x\right\rfloor \right)}{yT}
\]
Now we consider the set
\[
\Omega_{2}=\bigcup_{p,m:\,p^{m}\leq\left\lfloor x\right\rfloor }\left\{ t\in\left(2,\left\lfloor x\right\rfloor \right):\,\frac{p^{m}T}{T+1}\leq t\leq\frac{p^{m}T}{T-1},\,t\neq p^{m}\right\} 
\]
where $p$ runs over primes and $m\geq1$ are integers. We also assume
that, if $p^{m}=2$ or $p^{m}=\left\lfloor x\right\rfloor $, then
the intervals to consider are 
\[
2<t\leq\frac{2T}{T-1},\frac{\left\lfloor x\right\rfloor T}{T+1}\leq t<\left\lfloor x\right\rfloor ,
\]
respectively. So we can observe that
\[
t\in\Omega_{2}\Leftrightarrow\frac{t}{T\left|t-p^{m}\right|}\geq1
\]
then
\[
\frac{\left|\alpha\right|}{y}\int_{\Omega_{2}}\log\left(t\right)\min\left(1,\frac{t}{T\left\langle t\right\rangle }\right)\left(1-\frac{t}{y}\right)^{\mathrm{Re}\left(\alpha\right)-1}dt\ll\frac{\left|\alpha\right|}{y}M\left(\alpha,y,x\right)\sum_{p,m:\,p^{m}\leq\left\lfloor x\right\rfloor }\int_{p^{m}T/\left(T+1\right)}^{p^{m}T/\left(T-1\right)}\log\left(t\right)dt
\]
\[
\ll\frac{\left|\alpha\right|}{yT}M\left(\alpha,y,x\right)\sum_{p,m:\,p^{m}\leq\left\lfloor x\right\rfloor }p^{m}\log\left(p^{m}\right)\ll M\left(\alpha,y,x\right)\frac{\left|\alpha\right|\left\lfloor x\right\rfloor ^{2}}{yT}.
\]
Now we take
\[
\Omega_{3}=\left(2,\left\lfloor x\right\rfloor \right)^{*}\setminus\left\{ \Omega_{1}\cup\Omega_{2}\right\} 
\]
where
\[
\left(2,\left\lfloor x\right\rfloor \right)^{*}=\left(2,\left\lfloor x\right\rfloor \right)\setminus\left\{ m\geq1,\,p\,\mathrm{prime}:\,2<p^{m}<\left\lfloor x\right\rfloor \right\} .
\]
Obviously we can observe that if $t\in\Omega_{3}$ then we have to
consider the intervals
\begin{equation}
p^{m}-1<t<\frac{p^{m}T}{T+1},\label{eq:primi caso 1}
\end{equation}
\begin{equation}
\frac{p^{m}T}{T-1}<t<p^{m}+1\label{primi caso 2}
\end{equation}
if $p^{m}-1$ or $p^{m}+1$ are not prime powers, respectively and
\begin{equation}
\frac{\left(p^{m}-1\right)T}{T-1}<t<\frac{p^{m}T}{T+1},\label{primi caso 3}
\end{equation}
\begin{equation}
\frac{p^{m}T}{T-1}<t<\frac{\left(p^{m}+1\right)T}{T+1}\label{primi caso 4}
\end{equation}
if $p^{m}-1$ or $p^{m}+1$ are prime powers, respectively. If \eqref{eq:primi caso 1}
holds then
\[
\frac{\left|\alpha\right|}{yT}\int_{p^{m}-1}^{p^{m}T/\left(T+1\right)}\frac{t\log\left(t\right)}{p^{m}-t}dt\ll\frac{\left|\alpha\right|}{yT}p^{m}\log\left(p^{m}\right)\log\left(\frac{p^{m}}{T}\right)
\]
and the same bound holds for \eqref{primi caso 2}. If \eqref{primi caso 3}
holds then
\[
\frac{\left|\alpha\right|}{yT}\int_{\left(p^{m}-1\right)T/\left(T-1\right)}^{p^{m}T/\left(T+1\right)}\frac{t\log\left(t\right)}{\min\left(p^{m}-t,t-p^{m}+1\right)}dt\ll\frac{\left|\alpha\right|}{yT}p^{m}\log\left(p^{m}\right)\log\left(\frac{p^{m}}{T}\right)
\]
and the same holds for \eqref{primi caso 4}.

Summing up, we get
\[
\frac{\left|\alpha\right|}{y}\int_{\Omega_{3}}\log\left(t\right)\min\left(1,\frac{t}{T\left\langle t\right\rangle }\right)\left(1-\frac{t}{y}\right)^{\mathrm{Re}\left(\alpha\right)-1}dt\ll M\left(\alpha,y,x\right)\frac{\left|\alpha\right|\left\lfloor x\right\rfloor ^{2}\log^{2}\left(\left\lfloor x\right\rfloor T\right)}{yT}
\]
so finally we can write
\begin{equation}
\frac{\left|\alpha\right|}{y}\int_{2}^{\left\lfloor x\right\rfloor }\log\left(t\right)\min\left(1,\frac{t}{T\left\langle t\right\rangle }\right)\left(1-\frac{t}{y}\right)^{\mathrm{Re}\left(\alpha\right)-1}dt\ll M\left(\alpha,y,x\right)\frac{\left|\alpha\right|\left\lfloor x\right\rfloor ^{2}\log^{2}\left(T\left\lfloor x\right\rfloor \right)}{yT}.\label{int error term 2}
\end{equation}
To finish the proof we have only to substitute the $\psi\left(\left\lfloor x\right\rfloor \right)$
term in \eqref{eq:psi floor} with \eqref{troncata}, recalling that,
if $x$ is an integer, then the error term in \eqref{troncata} can
be written as $O\left(x\log^{2}\left(xT\right)/T\right)$ since $\left\langle x\right\rangle \geq1$.
\end{proof}

\section{Proof of Theorem 1}

Let $N>2$ be an integer and let $\psi\left(N\right)=\sum_{n\leq N}\Lambda\left(n\right)$
the Chebyshev psi function. From the identity
\[
\left(\sum_{m=0}^{k}a_{m}\right)\left(\sum_{m=0}^{k}b_{m}\right)=\sum_{m=0}^{2k}\sum_{h=0}^{m}a_{h}b_{m-h}-\sum_{m=0}^{k-1}\left(a_{m}\sum_{h=k+1}^{2k-m}b_{h}+b_{m}\sum_{h=k+1}^{2k-m}a_{h}\right),
\]
which can be proved observing that the set of lattice points 
\[
\left\{ \left(i,j\right):\,0\leq i+j\leq2N,\,0\leq i\leq N,\,0\leq j\leq N\right\} 
\]
forms a triangle that can be seen as a $N\times N$ square joint the
two triangles
\[
\left\{ \left(i,j\right):\,0\leq i\leq N-1,\,N+1\leq j\leq2N-i\right\} 
\]
and
\[
\left\{ \left(i,j\right):\,N+1\leq i\leq2N,\,0\leq j\leq2N-i\right\} ,
\]
we have that
\[
\sum_{n\leq N}\Lambda\left(n\right)\sum_{m\leq N}\Lambda\left(m\right)=\sum_{n\leq2N}\underset{{\scriptstyle m_{1}+m_{2}=n}}{\sum_{m_{1},m_{2}\leq n}}\Lambda\left(m_{1}\right)\Lambda\left(m_{2}\right)-2\sum_{n\leq N-1}\left(\Lambda\left(n\right)\sum_{m=N+1}^{2N-n}\Lambda\left(m\right)\right)
\]
then
\begin{align}
\sum_{n\leq2N}r_{G}\left(n\right)= & \psi^{2}\left(N\right)+2\sum_{n\leq N-1}\Lambda\left(n\right)\left(\psi\left(2N-n\right)-\psi\left(N\right)\right)\nonumber \\
= & 2\sum_{n\leq N}\Lambda\left(n\right)\psi\left(2N-n\right)-\psi\left(N\right)^{2}\label{eq:identit=0000E0 base}
\end{align}
so we will find the explicit formula for $\sum_{n\leq2N}r_{G}\left(n\right)$
using the classical explicit formula of $\psi_{0}\left(N\right)=\psi\left(N\right)-\frac{\Lambda\left(N\right)}{2}$
(for a reference see \cite{Daven}, chapter $17$). It is quite simple
to observe that the most delicate term to evaluate is 
\begin{equation}
2\sum_{n\leq N}\Lambda\left(n\right)\psi\left(2N-n\right).\label{eq:somma contosa}
\end{equation}
From the explicit formula of $\psi_{0}\left(N\right)$ we have that
\[
\psi\left(2N-n\right)=2N-n-\sum_{\rho}\frac{\left(2N-n\right)^{\rho}}{\rho}-\frac{\zeta^{\prime}}{\zeta}\left(0\right)-\frac{1}{2}\log\left(1-\frac{1}{\left(2N-n\right)^{2}}\right)+\frac{\Lambda\left(2N-n\right)}{2}
\]
so we now evaluate \eqref{eq:somma contosa} term by term.

\subsection{The main term}

From the Abel summation formula we have
\begin{equation}
2\sum_{n\leq N}\Lambda\left(n\right)\left(2N-n\right)=2N\psi\left(N\right)+2\psi_{1}\left(N\right)\label{pezzo con psi 1}
\end{equation}
where
\begin{equation}
\psi_{1}\left(N\right)=\int_{0}^{N}\psi\left(t\right)dt\label{psi 1}
\end{equation}
so from Theorem 28 of \cite{Ing} we get
\begin{equation}
2\sum_{n\leq N}\Lambda\left(n\right)\left(2N-n\right)=2N\psi\left(N\right)+N^{2}-2\sum_{\rho}\frac{N^{\rho+1}}{\rho\left(\rho+1\right)}-2N\frac{\zeta^{\prime}}{\zeta}\left(0\right)+2\frac{\zeta^{\prime}}{\zeta}\left(-1\right)-\sum_{r\geq1}\frac{N^{-2r+1}}{r\left(2r-1\right)}.\label{eq:formula termine principale}
\end{equation}
Furthermore it is not difficult to see that, expanding $2N\psi\left(N\right)$,
we have the asymptotic formula
\begin{equation}
2\sum_{n\leq N}\Lambda\left(n\right)\left(2N-n\right)=3N^{2}-2\sum_{\rho}\frac{N^{\rho+1}}{\rho\left(\rho+1\right)}+2\sum_{\rho}\frac{N^{\rho+1}}{\rho}+N\Lambda\left(N\right)+O\left(N\right)\label{eq:formula termine principale troncata}
\end{equation}
as $N\rightarrow\infty.$ 

\subsection{The term involving the series over the non-trivial zeros of $\zeta\left(s\right)$}

We now consider
\[
-2\sum_{n\leq N}\Lambda\left(n\right)\sum_{\rho}\frac{\left(2N-n\right)^{\rho}}{\rho}=-2\sum_{\rho}\frac{\left(2N\right)^{\rho}}{\rho}\sum_{n\leq N}\Lambda\left(n\right)\left(1-\frac{n}{2N}\right)^{\rho}
\]
then from Lemma \ref{lemma fondamentale}, taking $x=N,\,y=2N,\,\alpha=\rho$,
we get
\begin{align}
-2\sum_{n\leq N}\Lambda\left(n\right)\sum_{\rho}\frac{\left(2N-n\right)^{\rho}}{\rho}= & -2\sum_{\rho}\frac{\left(2N-2\right)^{\rho+1}}{\rho\left(\rho+1\right)}+2\sum_{\rho}\frac{N^{\rho+1}}{\rho\left(\rho+1\right)}-4\sum_{\rho}\frac{\left(2N-2\right)^{\rho}}{\rho}\label{eq:serie zeri1}\\
+ & 2\sum_{\rho_{1}}\left(2N\right)^{\rho_{1}}\left(\Gamma\left(\rho_{1}\right)\sum_{\rho_{2}}\frac{\left(2N\right)^{\rho_{2}}\Gamma\left(\rho_{2}\right)}{\Gamma\left(\rho_{1}+\rho_{2}+1\right)}\right.\label{eq:serie zeri2}\\
- & \left.\sum_{\rho_{2}}\frac{\left(2N\right)^{\rho_{2}}}{\rho_{2}}\left(B_{1/N}\left(\rho_{2}+1,\rho_{1}\right)+B_{1/2}\left(\rho_{1},\rho_{2}+1\right)\right)\right)\label{eq:serie zeri 5}\\
+ & 2\sum_{\rho_{1}}\frac{N^{\rho_{1}}}{\rho_{1}}\sum_{\rho_{2}}\frac{N^{\rho_{2}}}{\rho_{2}}+2\frac{\zeta^{\prime}}{\zeta}\left(0\right)\sum_{\rho}\frac{\left(2N-2\right)^{\rho}}{\rho}\label{eq:serie zeri3}\\
+ & \log\left(\frac{3}{4}\right)\sum_{\rho}\frac{\left(2N-2\right)^{\rho}}{\rho}-2\sum_{\rho}\frac{\left(2N\right)^{\rho}}{\rho}\omega\left(\rho,2N,N\right)-\Lambda\left(N\right)\sum_{\rho}\frac{N^{\rho}}{\rho}\label{serie zeri4}
\end{align}
where 
\begin{align*}
\omega\left(\rho,2N,N\right)= & -\frac{1}{2}\left(1-\frac{1}{2N}\right)^{\rho}\left(B_{(2N-2)/(2N-1)}\left(\rho+1,0\right)-B_{N/(2N-1)}\left(\rho+1,0\right)\right)\\
- & \frac{1}{2}\left(1+\frac{1}{2N}\right)^{\rho}\left(B_{(2N-2)/(2N+1)}\left(\rho+1,0\right)-B_{N/(2N+1)}\left(\rho+1,0\right)\right)\\
+ & B_{(2N-2)/2N}\left(\rho+1,0\right)-B_{1/2}\left(\rho+1,0\right).
\end{align*}
We can observe that the rearrangement in \eqref{eq:serie zeri1},
\eqref{eq:serie zeri2}, \eqref{eq:serie zeri 5}, \eqref{eq:serie zeri3}
and \eqref{serie zeri4} is legitimate: in \eqref{eq:serie zeri1},
\eqref{eq:serie zeri3} and \eqref{serie zeri4} the series are convergent
by the explicit formula of $\psi\left(N\right)$ (that is, in the
sense $\sum_{\rho}x^{\rho}/\rho=\lim_{T\rightarrow\infty}\sum_{\rho:\,\left|\gamma\right|\leq T}x^{\rho}/\rho$)
and $\sum_{\rho}\frac{\left(2N\right)^{\rho}}{\rho}\omega\left(\rho,2N,N\right)$
is convergent since, integrating by parts, we have, for all $0\leq h<1$,
that
\begin{equation}
\sum_{\rho}\frac{\left(2N\right)^{\rho}}{\rho}\left(1-\frac{1}{2N}\right)^{\rho}B_{h}\left(\rho+1,0\right)=\sum_{\rho}\frac{\left(2N\right)^{\rho}}{\rho}\left(1-\frac{1}{2N}\right)^{\rho}\int_{0}^{h}\frac{t^{\rho}}{1-t}dt\label{eq:serie zeri e beta inc1}
\end{equation}
\begin{equation}
=\sum_{\rho}\frac{\left(2N-1\right)^{\rho}}{\rho\left(\rho+1\right)}\left.\frac{t^{\rho+1}}{1-t}\right|_{0}^{h}+\sum_{\rho}\frac{\left(2N-1\right)^{\rho}}{\rho\left(\rho+1\right)}\int_{0}^{h}\frac{t^{\rho+1}}{\left(1-t\right)^{2}}dt\label{eq:serie zeri e beta 2}
\end{equation}
and so the convergence. This allow us to conclude that the double
series in \eqref{eq:serie zeri2} and \eqref{eq:serie zeri 5} is
convergent.

Now we want to give an estimation of some terms of \eqref{eq:serie zeri1},
\eqref{eq:serie zeri2}, \eqref{eq:serie zeri 5}, \eqref{eq:serie zeri3}
and \eqref{serie zeri4}. We start from the term
\[
\sum_{\rho}\frac{\left(2N-2\right)^{\rho}}{\rho}.
\]
Then, by the well known asymptotic
\begin{equation}
\psi\left(x\right)-x\ll\begin{cases}
x\exp\left(-C\sqrt{\log\left(x\right)}\right) & \mathrm{without}\,\mathrm{RH}\\
\sqrt{x}\log^{2}\left(x\right) & \mathrm{with}\,\mathrm{RH},
\end{cases}\,C>0,\,x>1\label{eq:stima zeri classica}
\end{equation}
(see \cite{Daven}, chapter $18$) we obtain
\[
\sum_{\rho}\frac{\left(2N-2\right)^{\rho}}{\rho}\ll\begin{cases}
N\exp\left(-C\sqrt{\log\left(N\right)}\right) & \mathrm{without}\,\mathrm{RH}\\
\sqrt{N}\log^{2}\left(N\right) & \mathrm{with}\,\mathrm{RH}
\end{cases}
\]
where $C>0$ is a real number. 

Now let us consider the terms 
\[
-2\frac{\zeta^{\prime}}{\zeta}\left(0\right)\sum_{\rho}\frac{\left(2N-2\right)^{\rho}}{\rho}+\log\left(\frac{3}{4}\right)\sum_{\rho}\frac{\left(2N-2\right)^{\rho}}{\rho}.
\]
We can easily see that the estimation
\[
-2\frac{\zeta^{\prime}}{\zeta}\left(0\right)\sum_{\rho}\frac{\left(2N-2\right)^{\rho}}{\rho}+\log\left(\frac{3}{4}\right)\sum_{\rho}\frac{\left(2N-2\right)^{\rho}}{\rho}\ll\begin{cases}
N\exp\left(-C\sqrt{\log\left(N\right)}\right) & \mathrm{without}\,\mathrm{RH}\\
\sqrt{N}\log^{2}\left(N\right) & \mathrm{with}\,\mathrm{RH}
\end{cases}
\]
holds. 

We now estimate the series
\[
\sum_{\rho}\frac{\left(2N\right)^{\rho}}{\rho}\omega\left(\rho,2N,N\right).
\]
We will consider only
\[
\sum_{\rho}\frac{\left(2N\right)^{\rho}}{\rho}\left(1-\frac{1}{2N}\right)^{\rho}\left(B_{(2N-2)/(2N-1)}\left(\rho+1,0\right)-B_{N/(2N-1)}\left(\rho+1,0\right)\right)
\]
since the other calculations are essentially the same. From the definition
of incomplete Beta function we have
\[
\sum_{\rho}\frac{\left(2N\right)^{\rho}}{\rho}\left(1-\frac{1}{2N}\right)^{\rho}\left(B_{(2N-2)/(2N-1)}\left(\rho+1,0\right)-B_{N/(2N-1)}\left(\rho+1,0\right)\right)
\]
\[
=\sum_{\rho}\frac{\left(2N-1\right)^{\rho}}{\rho}\int_{N/(2N-1)}^{(2N-2)/(2N-1)}\frac{t^{\rho}}{1-t}dt=\int_{N/(2N-1)}^{(2N-2)/(2N-1)}\frac{1}{1-t}\sum_{\rho}\frac{\left(2N-1\right)^{\rho}t^{\rho}}{\rho}dt
\]
since we know that we can exchange the series over the non-trivial
zeros and the integral. Now from \eqref{eq:stima zeri classica} we
get
\begin{equation}
\sum_{\rho}\frac{\left(2N\right)^{\rho}}{\rho}\left(1-\frac{1}{2N}\right)^{\rho}\left(B_{(2N-2)/(2N-1)}\left(\rho+1,0\right)-B_{N/(2N-1)}\left(\rho+1,0\right)\right)\label{eq:conv assoluta beta incompl}
\end{equation}
\[
\ll\begin{cases}
\left(2N-1\right)\int_{N/(2N-1)}^{(2N-2)/(2N-1)}t\exp\left(-C\sqrt{\log\left(t\left(2N-1\right)\right)}\right)/\left(1-t\right)dt & \mathrm{without}\,\mathrm{RH}\\
\sqrt{\left(2N-1\right)}\int_{N/(2N-1)}^{(2N-2)/(2N-1)}\sqrt{t}\log^{2}\left(t\left(2N-1\right)\right)/\left(1-t\right)dt & \mathrm{with}\,\mathrm{RH}
\end{cases}
\]
\[
\ll\begin{cases}
N\exp\left(-C\sqrt{\log\left(N\right)}\right) & \mathrm{without}\,\mathrm{RH}\\
\sqrt{N}\log^{3}\left(N\right) & \mathrm{with}\,\mathrm{RH}
\end{cases}
\]

where $C>0$ is a real constant. Hence we can conclude that
\begin{align}
-2\sum_{n\leq N}\Lambda\left(n\right)\sum_{\rho}\frac{\left(2N-n\right)^{\rho}}{\rho}= & -2\sum_{\rho}\frac{\left(2N-2\right)^{\rho+1}}{\rho\left(\rho+1\right)}+2\sum_{\rho}\frac{N^{\rho+1}}{\rho\left(\rho+1\right)}+2\sum_{\rho_{1}}\sum_{\rho_{2}}\frac{N^{\rho_{1}+\rho_{2}}}{\rho_{1}\rho_{2}}-\Lambda\left(N\right)\sum_{\rho}\frac{N^{\rho}}{\rho}\label{eq:serie zeri1-1}\\
+ & 2\sum_{\rho_{1}}\left(2N\right)^{\rho_{1}}\left(\Gamma\left(\rho_{1}\right)\sum_{\rho_{2}}\frac{\left(2N\right)^{\rho_{2}}\Gamma\left(\rho_{2}\right)}{\Gamma\left(\rho_{1}+\rho_{2}+1\right)}\right.\label{eq:serie zeri2-1}\\
- & \left.\sum_{\rho_{2}}\frac{\left(2N\right)^{\rho_{2}}}{\rho_{2}}\left(B_{1/N}\left(\rho_{2}+1,\rho_{1}\right)+B_{1/2}\left(\rho_{1},\rho_{2}+1\right)\right)\right)\label{eq:serie zeri3-1}\\
+ & O\left(\begin{cases}
N\exp\left(-C\sqrt{\log\left(N\right)}\right) & \mathrm{without}\,\mathrm{RH}\\
\sqrt{N}\log^{3}\left(N\right). & \mathrm{with}\,\mathrm{RH}
\end{cases}\right)
\end{align}

\subsection{The constant term}

Trivially
\[
-2\sum_{n\leq N}\Lambda\left(n\right)\frac{\zeta^{\prime}}{\zeta}\left(0\right)=-2\frac{\zeta^{\prime}}{\zeta}\left(0\right)\psi\left(N\right).
\]

\subsection{The term involving the logarithmic function}

By the Abel summation formula we have 
\[
-\sum_{n\leq N}\Lambda\left(n\right)\log\left(1-\frac{1}{\left(2N-n\right)^{2}}\right)=-\psi\left(N\right)\log\left(1-\frac{1}{N^{2}}\right)-\int_{2}^{N}\frac{2\psi\left(t\right)}{\left(2N-t\right)\left(2N-t-1\right)\left(2N-t+1\right)}dt
\]

and now using again the explicit formula for $\psi\left(t\right)$
we can evaluate term by term. Trivially we have that
\begin{align*}
-2\int_{2}^{N}\frac{tdt}{\left(2N-t\right)\left(2N-t-1\right)\left(2N-t+1\right)}= & \int_{2}^{N}\frac{2t}{2N-t}dt-\int_{2}^{N}\frac{t}{2N-t-1}dt-\int_{2}^{N}\frac{t}{2N-t+1}dt\\
= & 4+8N\mathrm{arccoth}\left(1-2N\right)+2N\left(\log\left(4\right)-1\right)\\
+ & 2N-2+(4N+2)\mathrm{arccoth}\left(\frac{3N}{2-N}\right)\\
- & 2\left(2N-1\right)\log\left(\frac{N-1}{2N-3}\right).
\end{align*}
For the second term, by Lemma \ref{lem:scambio formula expl}, we
observe that we can switch the integral with the series over the non-trivial
zeros so
\begin{align}
\int_{2}^{N}\frac{2dt}{\left(2N-t\right)\left(2N-t-1\right)\left(2N-t+1\right)}\sum_{\rho}\frac{t^{\rho}}{\rho}= & \sum_{\rho}\frac{1}{\rho}\int_{2}^{N}dt\left(\frac{-2t^{\rho}}{2N-t}+\frac{t^{\rho}}{2N-t-1}-\frac{t^{\rho}}{2N-t+1}\right)\nonumber \\
= & -2\sum_{\rho}\frac{\left(2N\right)^{\rho}}{\rho}\int_{1/N}^{1/2}\frac{u^{\rho}}{1-u}du\label{eq:calcolo integrale col zeri e log}\\
+ & \sum_{\rho}\frac{\left(2N-1\right)^{\rho}}{\rho}\int_{2/(2N-1)}^{N/\left(2N-1\right)}\frac{u^{\rho}}{1-u}du\label{zeri e log 3}\\
+ & \sum_{\rho}\frac{\left(2N+1\right)^{\rho}}{\rho}\int_{2/(2N+1)}^{N/\left(2N+1\right)}\frac{u^{\rho}}{1-u}du\label{zeri e log 2}
\end{align}
and the integrals in \eqref{eq:calcolo integrale col zeri e log},
\eqref{zeri e log 3} and \eqref{zeri e log 2} are difference of
two incomplete Beta functions. Observe that this arrangement is legitimate
since, arguing as in \eqref{eq:conv assoluta beta incompl} we can
prove that the series in \eqref{eq:calcolo integrale col zeri e log},
\eqref{zeri e log 3} and \eqref{zeri e log 2} converges absolutely.
Then we obviously get
\[
-2\frac{\zeta^{\prime}}{\zeta}\left(0\right)\int_{2}^{N}\frac{dt}{\left(2N-t\right)\left(2N-t-1\right)\left(2N-t+1\right)}=\frac{\zeta^{\prime}}{\zeta}\left(0\right)\left(\log\left(1-\frac{1}{\left(2N-2\right)^{2}}\right)-\log\left(1-\frac{1}{N^{2}}\right)\right).
\]
It remains to evaluate
\begin{align}
\int_{2}^{N}\frac{\log\left(1-1/t^{2}\right)}{\left(2N-t\right)\left(2N-t-1\right)\left(2N-t+1\right)}dt= & \int_{2}^{N}\frac{2\log\left(t\right)-\log\left(t-1\right)-\log\left(t+1\right)}{\left(2N-t\right)\left(2N-t-1\right)\left(2N-t+1\right)}dt\nonumber \\
= & -2\int_{2}^{N}\frac{2\log\left(t\right)-\log\left(t-1\right)-\log\left(t+1\right)}{2N-t}dt\label{log int 1}\\
+ & \int_{2}^{N}\frac{2\log\left(t\right)-\log\left(t-1\right)-\log\left(t+1\right)}{2N-t-1}dt\label{log int 2}\\
+ & \int_{2}^{N}\frac{2\log\left(t\right)-\log\left(t-1\right)-\log\left(t+1\right)}{2N-t+1}dt.\label{log int 3}
\end{align}
We will show only a single evaluation since the others are essentially
the same thing. We have that
\begin{align*}
\int_{2}^{N}\frac{\log\left(t\right)}{2N-t}dt= & \int_{1/N}^{1/2}\frac{\log\left(2N\right)+\log\left(u\right)}{1-u}du\\
= & \log\left(2N\right)\left(\log\left(\frac{1}{2}\right)-\log\left(1-\frac{1}{N}\right)\right)+\mathrm{Li}_{2}\left(\frac{1}{2}\right)-\mathrm{Li}_{2}\left(1-\frac{1}{N}\right)
\end{align*}
where $\mathrm{Li}_{2}\left(x\right)$ is the Dilogarithm function.
Using this strategy we will get, for all integrals in \eqref{log int 1},
\eqref{log int 2} and \eqref{log int 3}, a combination of elementary
functions and Dilogarithms.

Finally we note that
\[
-\sum_{n\leq N}\Lambda\left(n\right)\log\left(1-\frac{1}{\left(2N-n\right)^{2}}\right)\ll\frac{1}{N^{2}}\sum_{n\leq N}\Lambda\left(n\right)\ll\frac{1}{N}.
\]

\subsection{The term involving the Von Mangoldt function}

Lastly we have
\[
\sum_{n\leq N}\Lambda\left(n\right)\Lambda\left(2N-n\right)=\frac{1}{2}\sum_{n\leq2N}\Lambda\left(n\right)\Lambda\left(2N-n\right)=\frac{1}{2}\underset{{\scriptstyle m_{1}+m_{2}=2N}}{\sum_{m_{1},m_{2}\leq2N}}\Lambda\left(m_{1}\right)\Lambda\left(m_{2}\right)=\frac{r_{G}\left(2N\right)}{2}.
\]

\subsection{Put together all the pieces}

Finally we can rearrange all the parts. Expanding $\psi^{2}\left(N\right)$
in \eqref{eq:identit=0000E0 base} with its explicit formula and observing
that some terms cancel each other out (see, for example, \eqref{eq:formula termine principale troncata}
and \eqref{eq:serie zeri1}) we get that
\begin{align*}
\sideset{}{^{\prime}}\sum_{n\leq2N}r_{G}\left(n\right)= & 2N^{2}-2\sum_{\rho}\frac{\left(2N-2\right)^{\rho+1}}{\rho\left(\rho+1\right)}\\
+ & 2\sum_{\rho_{1}}\left(2N\right)^{\rho_{1}}\left(\Gamma\left(\rho_{1}\right)\sum_{\rho_{2}}\frac{\left(2N\right)^{\rho_{2}}\Gamma\left(\rho_{2}\right)}{\Gamma\left(\rho_{1}+\rho_{2}+1\right)}\right.\\
- & \sum_{\rho_{2}}\frac{\left(2N\right)^{\rho_{2}}}{\rho_{2}}\left(B_{1/N}\left(\rho_{2}+1,\rho_{1}\right)+B_{1/2}\left(\rho_{1},\rho_{2}+1\right)\right)+F\left(N\right)
\end{align*}
where $F\left(N\right)$ can be explicitly calculated in terms of
special functions like the incomplete Beta function and the Dilogarithm
and $F\left(N\right)=O\left(N\right)$ as $N\rightarrow\infty$.

\subsection{The truncated formula}

We want to prove the truncated version of the formula. We start taking
$T_{1}>2$ and substituting the formula
\begin{align*}
\psi\left(2N-n\right)= & 2N-n-\sum_{\rho:\,\left|\gamma\right|\leq T_{1}}\frac{\left(2N-n\right)^{\rho}}{\rho}-\frac{\zeta^{\prime}}{\zeta}\left(0\right)-\frac{1}{2}\log\left(1-\frac{1}{\left(2N-n\right)^{2}}\right)\\
+ & \frac{\Lambda\left(2N-n\right)}{2}+O\left(\frac{N\log^{2}\left(NT_{1}\right)}{T_{1}}\right)
\end{align*}
 in \eqref{eq:somma contosa}. We will evaluating the sum term by
term. Again we recall that, if $x$ is an integer, then the error
term in \eqref{troncata} can be write as $O\left(x\log^{2}\left(xT\right)/T\right)$
since $\left\langle x\right\rangle \geq1$.

\subsection{The main term of the truncated formula}

Following the $3.1$ section we get
\[
2\sum_{n\leq N}\Lambda\left(n\right)\left(2N-n\right)=2N\psi\left(N\right)+2\psi_{1}\left(N\right).
\]
From \eqref{troncata} we get
\[
2N\psi\left(N\right)=2N^{2}-2\sum_{\rho:\,\left|\gamma\right|\leq T_{2}}\frac{N^{\rho+1}}{\rho}-2N\frac{\zeta^{\prime}}{\zeta}\left(0\right)-N\log\left(1-\frac{1}{N^{2}}\right)+N\Lambda\left(N\right)+O\left(\frac{N^{2}\log^{2}\left(NT_{2}\right)}{T_{2}}\right)
\]
where $T_{2}>2$ will be choose later. For the evaluation of $\psi_{1}\left(N\right)$
we observe that
\[
2\psi_{1}\left(N\right)=2\int_{0}^{N}\psi\left(t\right)dt
\]
\[
=2\int_{2}^{N}\left(t-\sum_{\rho:\,\left|\gamma\right|\leq T_{3}}\frac{t^{\rho}}{\rho}-\frac{\zeta^{\prime}}{\zeta}\left(0\right)-\frac{\log\left(1-1/t^{2}\right)}{2}+O\left(\frac{t\log^{2}\left(T_{3}t\right)}{T_{3}}+\log\left(t\right)\min\left(1,\frac{t}{T_{3}\left\langle t\right\rangle }\right)\right)\right)dt
\]
\[
=N^{2}-2\sum_{\rho:\,\left|\gamma\right|\leq T_{3}}\frac{N^{\rho+1}}{\rho\left(\rho+1\right)}-2N\frac{\zeta^{\prime}}{\zeta}\left(0\right)-\sum_{r\geq1}\frac{N^{-2r+1}}{r\left(2r-1\right)}+O\left(\frac{N^{2}\log^{2}\left(T_{3}N\right)}{T_{3}}\right)
\]
where $T_{3}>2$ and, for the integration of the error term we used
the same strategy of \eqref{eq:int error term 1} and \eqref{int error term 2}.

\subsection{The term involving the series over the non-trivial zeros of $\zeta\left(s\right)$
of the truncated formula}

Following the $3.2$ section we have
\[
-2\sum_{n\leq N}\Lambda\left(n\right)\sum_{\rho:\,\left|\gamma\right|\leq T_{1}}\frac{\left(2N-n\right)^{\rho}}{\rho}=-2\sum_{\rho:\,\left|\gamma\right|\leq T_{1}}\frac{\left(2N\right)^{\rho}}{\rho}\sum_{n\leq N}\Lambda\left(n\right)\left(1-\frac{n}{2N}\right)^{\rho}
\]
then from Lemma \ref{lemma fondamentale}, taking $x=N,\,y=2N,\,\alpha=\rho$,
$T=T_{4}>2$ and observing that, in this case, we have
\[
M\left(\rho,2N,N\right)\ll1
\]
where the implicit constant does not depend on $N$ or $\rho$, we
get
\begin{align}
-2\sum_{n\leq N}\Lambda\left(n\right)\sum_{\rho:\,\left|\gamma\right|\leq T_{1}}\frac{\left(2N-n\right)^{\rho}}{\rho}= & -2\sum_{\rho:\,\left|\gamma\right|\leq T_{1}}\frac{\left(2N-2\right)^{\rho+1}}{\rho\left(\rho+1\right)}+2\sum_{\rho:\,\left|\gamma\right|\leq T_{1}}\frac{N^{\rho+1}}{\rho\left(\rho+1\right)}-4\sum_{\rho:\,\left|\gamma\right|\leq T_{1}}\frac{\left(2N-2\right)^{\rho}}{\rho}\nonumber \\
+ & 2\sum_{\rho_{1}:\,\left|\gamma_{1}\right|\leq T_{1}}\left(2N\right)^{\rho_{1}}\left(\Gamma\left(\rho_{1}\right)\sum_{\rho_{2}:\,\left|\gamma_{2}\right|\leq T_{4}}\frac{\left(2N\right)^{\rho_{2}}\Gamma\left(\rho_{2}\right)}{\Gamma\left(\rho_{1}+\rho_{2}+1\right)}\right.\nonumber \\
- & \left.\sum_{\rho_{2}:\,\left|\gamma_{2}\right|\leq T_{4}}\frac{\left(2N\right)^{\rho_{2}}}{\rho_{2}}\left(B_{1/N}\left(\rho_{2}+1,\rho_{1}\right)+B_{1/2}\left(\rho_{1},\rho_{2}+1\right)\right)\right)\nonumber \\
+ & 2\sum_{\rho_{1}:\,\left|\gamma_{1}\right|\leq T_{1}}\frac{N^{\rho_{1}}}{\rho_{1}}\sum_{\rho_{2}:\,\left|\gamma_{2}\right|\leq T_{4}}\frac{N^{\rho_{2}}}{\rho_{2}}+2\frac{\zeta^{\prime}}{\zeta}\left(0\right)\sum_{\rho:\,\left|\gamma\right|\leq T_{1}}\frac{\left(2N-2\right)^{\rho}}{\rho}\nonumber \\
+ & \log\left(\frac{3}{4}\right)\sum_{\rho:\,\left|\gamma\right|\leq T_{1}}\frac{\left(2N-2\right)^{\rho}}{\rho}-2\sum_{\rho:\,\left|\gamma\right|\leq T_{1}}\frac{\left(2N\right)^{\rho}}{\rho}\omega\left(\rho,2N,N\right)\nonumber \\
- & \Lambda\left(N\right)\sum_{\rho:\,\left|\gamma\right|\leq T_{1}}\frac{N^{\rho}}{\rho}+O\left(\frac{N\log^{2}\left(NT_{4}\right)}{T_{4}}\sum_{\rho:\,\left|\gamma\right|\leq T_{1}}\left(2N\right)^{\beta}\right)\label{eq:erroe serie doppia troncata}
\end{align}
where $T_{4}>2$ will be choose later. It is clear that the size of
\eqref{eq:erroe serie doppia troncata} changes if we assume RH or
not.

\subsection{The constant term of the truncated formula}

Again we observe that
\begin{equation}
-2\sum_{n\leq N}\Lambda\left(n\right)\frac{\zeta^{\prime}}{\zeta}\left(0\right)=-2\frac{\zeta^{\prime}}{\zeta}\left(0\right)\psi\left(N\right)\label{eq:termine costante formula troncata}
\end{equation}
and so we can substitute \eqref{troncata} in \eqref{eq:termine costante formula troncata}
with some $T_{5}>2$ that will be choose later.

\subsection{The term involving the logarithmic function of the truncated formula}

By section $3.4$ we know that By the Abel summation formula we have
\[
-\sum_{n\leq N}\Lambda\left(n\right)\log\left(1-\frac{1}{\left(2N-n\right)^{2}}\right)=-\psi\left(N\right)\log\left(1-\frac{1}{N^{2}}\right)-\int_{2}^{N}\frac{2\psi\left(t\right)}{\left(2N-t\right)\left(2N-t-1\right)\left(2N-t+1\right)}dt.
\]
We fix $T_{6}>2.$ Then we can expand the term
\[
\psi\left(N\right)\log\left(1-\frac{1}{N^{2}}\right)
\]
with \eqref{troncata}. The integral
\[
\int_{2}^{N}\frac{2\psi\left(t\right)}{\left(2N-t\right)\left(2N-t-1\right)\left(2N-t+1\right)}dt
\]
will be treated as in $3.4$ with but, fixing $T_{7}>2,$ we will
get the extra terms
\[
\frac{1}{T}_{7}\int_{2}^{N}\frac{t\log^{2}\left(tT_{7}\right)}{\left(2N-t\right)\left(2N-t-1\right)\left(2N-t+1\right)}dt\ll\frac{\log^{2}\left(NT_{7}\right)}{NT_{7}}
\]
and
\[
\int_{2}^{N}\frac{\log\left(t\right)\min\left(1,t/\left(T_{7}\left\langle t\right\rangle \right)\right)}{\left(2N-t\right)\left(2N-t-1\right)\left(2N-t+1\right)}dt\ll\frac{\log^{2}\left(N\right)}{T_{7}}
\]
arguing as in \eqref{int error term 2}. 

\subsection{The error term and the Von Mangoldt term}

Trivially we have 
\[
\frac{1}{T_{1}}\sum_{n\leq N}\Lambda\left(n\right)\left(2N-n\right)\log^{2}\left(T_{1}\left(2N-n\right)\right)\ll\frac{N^{2}\log^{2}\left(T_{1}N\right)}{T_{1}}
\]
and the ``Von Mangoldt term'' is exactly as in $3.5$.

\subsection{Put together all the pieces of the truncated formula}

Now it remains to expand the term $\psi^{2}\left(N\right)$ in \eqref{eq:identit=0000E0 base}
with \eqref{troncata} fixing some $T_{8}>2.$ We want to exploit
the cancellation of this formula so we have to choose carefully the
$T_{j}$ terms. Obviously if the take $T_{j}\rightarrow\infty$ in
a suitable order we can recognize the previous formula. The choice
of $T_{j}$ is very delicate; we must take advantage of the cancellation
effectively but we do not want to take too large parameters. To finish
our version of the formula we have to impose the condition
\[
T_{j}=T^{\prime},\,j=1,\dots,8,\,j\neq4;
\]
this assumption guarantees, for example, the cancellation of sums
like
\[
-2\sum_{\rho:\,\left|\gamma\right|\leq T_{3}}\frac{N^{\rho+1}}{\rho\left(\rho+1\right)}+2\sum_{\rho:\,\left|\gamma\right|\leq T_{1}}\frac{N^{\rho+1}}{\rho\left(\rho+1\right)}
\]
(see section $3.8$ and $3.9$) and
\[
-2\sum_{\rho:\,\left|\gamma\right|\leq T_{2}}\frac{N^{\rho+1}}{\rho}+2\sum_{\rho:\,\left|\gamma\right|\leq T_{8}}\frac{N^{\rho+1}}{\rho}
\]
(for the first sum see section $3.8$, for the second sum just expand
$\psi^{2}\left(N\right)$ with its truncated formula). Then we take
\[
T_{4}=T^{\prime\prime}
\]
and we estimate the sum in \eqref{eq:erroe serie doppia troncata}
with the Riemann - Von Mangoldt formula and the zero free region of
$\zeta\left(s\right).$

\section{Some Remarks}

We now present some remark of this result. The first remark is that
this result can be improved: with a bit of work it is not difficult
to extract the term $N\cdot\mathrm{constant}$ explicitly and give
a lower error of the asymptotic (which will depends on the RH assumption).
The second remark is that from Lemma \ref{lemma fondamentale} we
can, in principle, find the explicit formula for the Ces\`aro average
of Goldbach numbers
\[
\frac{1}{\Gamma\left(k+1\right)}\sum_{n\leq N}r_{G}\left(n\right)\left(N-n\right)^{k},\,k>0.
\]
We use the words ``in principle'' because we will expect a lot of
terms to calculate. The idea is the following: from the identity
\begin{equation}
\frac{1}{\Gamma\left(k+1\right)}\sum_{n\leq N}r_{G}\left(n\right)\left(N-n\right)^{k}=\frac{N^{k}}{\Gamma\left(k+1\right)}\sum_{n<N}\Lambda\left(n\right)\left(1-\frac{n}{N}\right)^{k}\sum_{m<N-n}\Lambda\left(m\right)\left(1-\frac{m}{N-n}\right)^{k}\label{peso ces=0000E0ro}
\end{equation}
(observed in \cite{Perelli}) we can easily see that the problem boils
down to evaluate the combination of sums involving the Von Mangoldt
function with a Ces\`aro weight. So we can substitute in \eqref{peso ces=0000E0ro}
the explicit formula with $y=N-n,\,x=N-n-1$ and $\alpha=k$ and evaluate
the sum term by term. For example the main term (that we know is $\frac{N^{k+2}}{\Gamma\left(k+3\right)}$
from \cite{Perelli}) will be come from
\[
\frac{N^{k}}{\Gamma\left(k+1\right)}\sum_{n<N}\Lambda\left(n\right)\left(1-\frac{n}{N}\right)^{k}\frac{N-n}{k+1}=\frac{1}{\Gamma\left(k+2\right)}\sum_{n\leq N}\Lambda\left(n\right)\left(N-n\right)^{k+1}.
\]
To confirm our claim note that by the Abel summation formula we find
that
\begin{equation}
\frac{N^{k}}{\Gamma\left(k+1\right)}\sum_{n<N}\Lambda\left(n\right)\left(1-\frac{n}{N}\right)^{k}\frac{N-n}{k+1}=\frac{k+1}{\Gamma\left(k+2\right)}\int_{2}^{N}\psi\left(t\right)\left(N-t\right)^{k}dt\label{main term peso ces=0000E0ro}
\end{equation}
then substituting the explicit formula for $\psi\left(t\right)$ in
\eqref{main term peso ces=0000E0ro} we will find the main term of
the explicit formula plus other terms. In fact we can see that
\begin{align*}
\frac{k+1}{\Gamma\left(k+2\right)}\int_{2}^{N}t\left(N-t\right)^{k}dt= & \frac{k+1}{\Gamma\left(k+2\right)}N^{k+2}\int_{2/N}^{1}u\left(1-u\right)^{k}du\\
= & \frac{N^{k+2}}{\Gamma\left(k+3\right)}+H_{k}\left(N\right)
\end{align*}
where $H_{k}\left(N\right)=O_{k}\left(N^{k+1}\right)$, as expected.

\section{Proof of Theorem 4}

Now we show that a very similar approach to the previous one can be
used also to find the explicit form of the average of primes in tuples.
We start again with a summation identity:
\[
\sum_{h=0}^{M}\sum_{n=0}^{N}a_{n}b_{n+h}=\left(\sum_{n=0}^{N}a_{n}\right)\left(\sum_{n=0}^{N+M}b_{n}\right)-\sum_{n=0}^{N-1}\left(b_{n}\sum_{m=n+1}^{N}a_{m}+a_{n}\sum_{m=n+M+1}^{N+M}b_{m}\right)
\]
where $M,\,N\geq0$ are integers, which can be proved observing that
the set of lattice points $\left\{ \left(i,i+j\right):0\leq i\leq N,0\leq j\leq M\right\} $
forms a parallelogram, which can be seen as a $N\times\left(N+M\right)$
rectangular minus the triangles 
\[
\left\{ \left(i,j\right):0\leq j\leq N-1,j+1\leq i\leq N\right\} ,\,\left\{ \left(i,j\right):0\leq i\leq N-1,i+1+M\leq j\leq M+N\right\} .
\]
We now fix $N>2$ and $0\leq M\leq N$ and define
\[
r_{PT}\left(N,h\right)=\sum_{n=0}^{N}\Lambda\left(n\right)\Lambda\left(n+h\right).
\]
 We have that
\begin{align}
\sum_{h=0}^{M}r_{PT}\left(N,h\right)= & \left(\sum_{n=0}^{N}\Lambda\left(n\right)\right)\left(\sum_{n=0}^{N+M}\Lambda\left(n\right)\right)-\sum_{n=0}^{N-1}\left(\Lambda\left(n\right)\sum_{m=n+1}^{N}\Lambda\left(m\right)+\Lambda\left(n\right)\sum_{m=n+M+1}^{N+M}\Lambda\left(m\right)\right)\nonumber \\
= & \psi\left(N\right)\psi\left(N+M\right)-\sum_{n=0}^{N-1}\Lambda\left(n\right)\left(\psi\left(N\right)-\psi\left(n\right)\right)-\sum_{n=0}^{N-1}\Lambda\left(n\right)\left(\psi\left(N+M\right)-\psi\left(n+M\right)\right)\nonumber \\
= & \sum_{n\leq N}\Lambda\left(n\right)\psi\left(n+M\right)+\sum_{n\leq N}\Lambda\left(n\right)\psi\left(n\right)-\psi^{2}\left(N\right).\label{eq:prime in tuples id base}
\end{align}
Again we will consider
\[
\sum_{n\leq N}\Lambda\left(n\right)\psi\left(n+M\right)+\sum_{n\leq N}\Lambda\left(n\right)\psi\left(n\right)
\]
and we will substitute $\psi\left(x\right)$ with its explicit formula.

\subsection{The main term}

Substituting $\psi\left(x\right)$ with $x$ we get
\[
\sum_{n\leq N}\Lambda\left(n\right)\left(n+M\right)+\sum_{n\leq N}n\Lambda\left(n\right)
\]
\[
=2\sum_{n\leq N}n\Lambda\left(n\right)+M\psi\left(N\right).
\]
By the Abel summation formula we have that
\[
2\sum_{n\leq N}n\Lambda\left(n\right)=2\psi\left(N\right)N-2\psi_{1}\left(N\right),
\]
where $\psi_{1}\left(N\right)$ is \eqref{psi 1}. So, expanding $\psi\left(N\right)$
and $\psi_{1}\left(N\right)$ with their explicit formulae, we obtain
\begin{align*}
2\sum_{n\leq N}n\Lambda\left(n\right)+M\psi\left(N\right)= & N^{2}+NM-2\sum_{\rho}\frac{N^{\rho+1}}{\rho}\\
+ & 2\sum_{\rho}\frac{N^{\rho+1}}{\rho\left(\rho+1\right)}-M\sum_{\rho}\frac{N^{\rho}}{\rho}\\
- & M\frac{\zeta^{\prime}}{\zeta}\left(0\right)-2\frac{\zeta^{\prime}}{\zeta}\left(-1\right)-\left(N+\frac{M}{2}\right)\log\left(1-\frac{1}{N^{2}}\right)\\
+ & \sum_{r\geq1}\frac{N^{-2r+1}}{r\left(2r-1\right)}+\left(N+\frac{M}{2}\right)\Lambda\left(N\right).
\end{align*}
Obviously, from \eqref{eq:stima zeri classica}, we can also see that
\begin{align}
2\sum_{n\leq N}n\Lambda\left(n\right)+M\psi\left(N\right)= & N^{2}+NM-2\sum_{\rho}\frac{N^{\rho+1}}{\rho}+2\sum_{\rho}\frac{N^{\rho+1}}{\rho\left(\rho+1\right)}\nonumber \\
+ & \left(N+\frac{M}{2}\right)\Lambda\left(N\right)+O\left(E\left(M,N\right)\right)\label{errore main term tuples}
\end{align}
where
\[
E\left(M,N\right)=\begin{cases}
N\left(M+1\right)\exp\left(-C\sqrt{\log\left(N\right)}\right) & \mathrm{without}\,\mathrm{RH}\\
\sqrt{N}\left(M+1\right)\log^{2}\left(N\right) & \mathrm{with}\,\mathrm{RH}
\end{cases},\,C>0,
\]
and where the implicit constant in \eqref{errore main term tuples}
does not depend on $M$.

\subsection{The term involving the series over the non-trivial zeros of $\zeta\left(s\right)$}

We have now to evaluate
\[
-\sum_{n\leq N}\Lambda\left(n\right)\sum_{\rho}\frac{\left(n+M\right)^{\rho}}{\rho}-\sum_{n\leq N}\Lambda\left(n\right)\sum_{\rho}\frac{n^{\rho}}{\rho}.
\]
We can consider only the sum
\[
-\sum_{n\leq N}\Lambda\left(n\right)\sum_{\rho}\frac{\left(n+M\right)^{\rho}}{\rho}=-\sum_{\rho}\frac{1}{\rho}\sum_{n\leq N}\Lambda\left(n\right)\left(n+M\right)^{\rho}
\]
since the other is the same sum with the assumption $M=0.$ Again
from the Abel summation formula we obtain
\begin{equation}
\sum_{n\leq N}\Lambda\left(n\right)\left(n+M\right)^{\rho}=\psi\left(N\right)\left(N+M\right)^{\rho}-\rho\int_{2}^{N}\psi\left(t\right)\left(t+M\right)^{\rho-1}dt.\label{eq:serie zeri tuples}
\end{equation}
Substituting the main term of the explicit formula of $\psi\left(t\right)$
in \eqref{eq:serie zeri tuples} we obtain
\begin{align*}
-\rho\int_{2}^{N}t\left(t+M\right)^{\rho-1}dt= & -N\left(N+M\right)^{\rho}+2\left(2+M\right)^{\rho}\\
+ & \frac{\left(N+M\right)^{\rho+1}}{\rho+1}-\frac{\left(2+M\right)^{\rho+1}}{\rho+1}.
\end{align*}
Obviously if $M=0$ we can make the same calculations. Now we consider
the sum over the non-trivial zeros. Assume that $M>0.$ Then by Lemma
\ref{lem:scambio formula expl} we have
\begin{align}
\rho_{1}\int_{2}^{N}\sum_{\rho_{2}}\frac{t^{\rho_{2}}}{\rho_{2}}\left(t+M\right)^{\rho_{1}-1}dt= & \rho_{1}\sum_{\rho_{2}}\frac{1}{\rho_{2}}\int_{2}^{N}t^{\rho_{2}}\left(t+M\right)^{\rho_{1}-1}dt\nonumber \\
= & \rho_{1}\sum_{\rho_{2}}\frac{M^{\rho_{1}+\rho_{2}}}{\rho_{2}}\int_{2/M}^{N/M}u^{\rho_{2}}\left(1+u\right)^{\rho_{1}-1}du\nonumber \\
= & \rho_{1}\sum_{\rho_{2}}\frac{M^{\rho_{1}+\rho_{2}}\left(-1\right)^{\rho_{2}+1}}{\rho_{2}}\left(B_{-N/M}\left(\rho_{2}+1,\rho_{1}\right)-B_{-2/M}\left(\rho_{2}+1,\rho_{1}\right)\right)\label{incomplete beta con x negativo}
\end{align}
where in \eqref{incomplete beta con x negativo} we extended the definition
of incomplete Beta function to a negative integration domain (or,
if we prefer, we can write the integral in terms of the Gauss Hypergeometric
function $_{2}F_{1}\left(a,b;c;z\right)$). In the other case (or
if $M=0$) we get
\begin{align*}
\rho_{1}\int_{2}^{N}\sum_{\rho_{2}}\frac{t^{\rho_{2}}}{\rho_{2}}t^{\rho_{1}-1}dt= & \rho_{1}\sum_{\rho_{2}}\frac{1}{\rho_{2}}\int_{2}^{N}t^{\rho_{1}+\rho_{2}-1}dt\\
= & \rho_{1}\sum_{\rho_{2}}\frac{N^{\rho_{1}+\rho_{2}}-2^{\rho_{1}+\rho_{2}}}{\rho_{2}\left(\rho_{1}+\rho_{2}\right)}.
\end{align*}

Then we have to consider the constant term
\begin{align*}
\rho\frac{\zeta^{\prime}}{\zeta}\left(0\right)\int_{2}^{N}\left(t+M\right)^{\rho-1}dt= & \frac{\zeta^{\prime}}{\zeta}\left(0\right)\left(N+M\right)^{\rho}-\frac{\zeta^{\prime}}{\zeta}\left(0\right)\left(2+M\right)^{\rho}
\end{align*}
and lastly
\begin{align}
\frac{\rho}{2}\int_{2}^{N}\log\left(1-\frac{1}{t^{2}}\right)\left(t+M\right)^{\rho-1}dt= & \frac{1}{2}\log\left(1-\frac{1}{N^{2}}\right)\left(N+M\right)^{\rho}\nonumber \\
- & \frac{1}{2}\log\left(1-\frac{1}{4}\right)\left(2+M\right)^{\rho}\\
+ & \int_{2}^{N}\frac{\left(t+M\right)^{\rho-1}}{t}dt-\frac{1}{2}\int_{2}^{N}\frac{\left(t+M\right)^{\rho-1}}{t-1}dt\label{integrali pallosi}\\
- & \frac{1}{2}\int_{2}^{N}\frac{\left(t+M\right)^{\rho-1}}{t+1}dt\label{eq:integrali pallosi 2}
\end{align}
and the integrals in \eqref{integrali pallosi} and \ref{eq:integrali pallosi 2}
can be evaluated as a difference of two incomplete Beta functions
with negative integration domain. For example
\[
\int_{2}^{N}\frac{\left(t+M\right)^{\rho}}{t}dt=M^{\rho}\left(-1\right)^{\rho}\int_{-2/M}^{-N/M}\frac{\left(1-u\right)^{\rho}}{u}du
\]
\[
=M^{\rho}\left(-1\right)^{\rho}\lim_{\epsilon\rightarrow0^{+}}\left(B_{-2/M}\left(\epsilon,\rho+1\right)-B_{-N/M}\left(\epsilon,\rho+1\right)\right)
\]
(if one prefer this integral can be written as a combination of Gauss
Hypergeometric function). If $M=0$ we can do a similar calculation.
It is more interesting to note that, in the form of \eqref{integrali pallosi}
and \eqref{eq:integrali pallosi 2}, we can easily evaluate the integral
since, summing up, we have
\begin{align}
\sum_{\rho}\frac{\rho}{2}\int_{2}^{N}\log\left(1-\frac{1}{t^{2}}\right)\left(t+M\right)^{\rho-1}dt= & \frac{1}{2}\log\left(1-\frac{1}{N^{2}}\right)\sum_{\rho}\frac{\left(N+M\right)^{\rho}}{\rho}\nonumber \\
- & \frac{1}{2}\log\left(1-\frac{1}{4}\right)\sum_{\rho}\frac{\left(2+M\right)^{\rho}}{\rho}\\
- & \int_{2}^{N}\sum_{\rho}\frac{\left(t+M\right)^{\rho}}{\rho}\frac{1}{t\left(t-1\right)\left(t+1\right)}dtt\label{integrali pallosi-1}\\
\ll & \begin{cases}
N\exp\left(-C\sqrt{\log\left(N\right)}\right), & \mathrm{without}\,\mathrm{RH}\\
\sqrt{N}\log^{2}\left(N\right), & \mathrm{with}\,\mathrm{RH}
\end{cases}\label{eq:stima roba pallosa}
\end{align}
and the implicit constant does not depend on $M.$ So, expand $\psi\left(N\right)$
in \eqref{eq:serie zeri tuples} with its explicit formula and observing
that some terms cancel each other out, we finally get 
\begin{align*}
-\sum_{n\leq N}\Lambda\left(n\right)\sum_{\rho}\frac{\left(n+M\right)^{\rho}}{\rho}-\sum_{n\leq N}\Lambda\left(n\right)\sum_{\rho}\frac{n^{\rho}}{\rho}= & -\sum_{\rho}\frac{\left(N+M\right)^{\rho+1}}{\rho\left(\rho+1\right)}+\sum_{\rho}\frac{\left(2+M\right)^{\rho+1}}{\rho\left(\rho+1\right)}-\\
- & 2\sum_{\rho}\frac{\left(2+M\right)^{\rho}}{\rho}-\sum_{\rho}\frac{N^{\rho+1}}{\rho\left(\rho+1\right)}\\
- & F_{1}\left(N,M\right)-F_{2}\left(N,M\right)\\
+ & \sum_{\rho_{1}}\frac{N^{\rho_{1}}}{\rho_{1}}\sum_{\rho_{2}}\frac{\left(N+M\right)^{\rho_{2}}}{\rho_{2}}+\sum_{\rho_{1}}\frac{N^{\rho_{1}}}{\rho_{1}}\sum_{\rho_{2}}\frac{N^{\rho_{2}}}{\rho_{2}}\\
- & \frac{\Lambda\left(N\right)}{2}\left(\sum_{\rho}\frac{\left(N+M\right)^{\rho}}{\rho}+\sum_{\rho}\frac{N^{\rho}}{\rho}\right)+F_{3}\left(N,M\right)
\end{align*}
where:
\[
F_{1}\left(N,M\right)=\left(\sum_{\rho_{1}}\sum_{\rho_{2}}\frac{M^{\rho_{1}+\rho_{2}}\left(-1\right)^{\rho_{2}+1}}{\rho_{2}}\left(B_{-N/M}\left(\rho_{2}+1,\rho_{1}\right)-B_{-2/M}\left(\rho_{2}+1,\rho_{1}\right)\right)\right)\cdot1\left(M\right),
\]
\[
F_{2}\left(N,M\right)=\left(\sum_{\rho_{1}}\sum_{\rho_{2}}\frac{N^{\rho_{1}+\rho_{2}}-2^{\rho_{1}+\rho_{2}}}{\rho_{2}\left(\rho_{1}+\rho_{2}\right)}\right)\cdot\widetilde{1}\left(M\right),
\]
\[
1\left(M\right)=\begin{cases}
0, & M=0\\
1, & M>0,
\end{cases}
\]
\[
\widetilde{1}\left(M\right)=\begin{cases}
2, & M=0\\
1, & M>0
\end{cases}
\]
and $F_{3}\left(N,M\right)$ can be explicitly calculated in terms
of the incomplete Beta function and with the property
\[
F_{3}\left(N,M\right)\ll\begin{cases}
N\exp\left(-C\sqrt{\log\left(N\right)}\right), & \mathrm{without}\,\mathrm{RH}\\
\sqrt{N}\log^{2}\left(N\right), & \mathrm{with}\,\mathrm{RH.}
\end{cases}
\]
Note that, arguing analogously to $3.2$, we can conclude that the
rearrangement is legitimate and the double series in $F_{1}\left(N,M\right)$
and $F_{2}\left(N,M\right)$ converges.

\subsection{The constant term}

Trivially we have
\[
-2\frac{\zeta^{\prime}}{\zeta}\left(0\right)\sum_{n\leq N}\Lambda\left(n\right)=-2\frac{\zeta^{\prime}}{\zeta}\left(0\right)\psi\left(N\right)
\]

\subsection{The term involving the logarithmic function and the Von Mangoldt
function}

We have now to evaluate
\[
-\frac{1}{2}\sum_{n\leq N}\Lambda\left(n\right)\left(\log\left(1-\frac{1}{\left(n+M\right)^{2}}\right)+\log\left(1-\frac{1}{n^{2}}\right)\right)
\]
\[
=-\frac{\psi\left(N\right)}{2}\left(\log\left(1-\frac{1}{\left(N+M\right)^{2}}\right)+\log\left(1-\frac{1}{N^{2}}\right)\right)
\]
\begin{equation}
+\int_{2}^{N}\frac{\psi\left(t\right)}{\left(M+t\right)\left(M+t+1\right)\left(M+t-1\right)}dt+\int_{2}^{N}\frac{\psi\left(t\right)}{t\left(t+1\right)\left(t-1\right)}dt\label{eq:primi gemelli log term}
\end{equation}
which can be evaluated again integrating term by term the explicit
formula of $\psi\left(t\right)$. Arguing as in the previous sections,
it is possible to calculate the integrals in \eqref{eq:primi gemelli log term}
in terms of elementary functions, incomplete Beta functions and Dilogarithms.
Furthermore
\[
-\frac{1}{2}\sum_{n\leq N}\Lambda\left(n\right)\left(\log\left(1-\frac{1}{\left(n+M\right)^{2}}\right)+\log\left(1-\frac{1}{n^{2}}\right)\right)\ll\sum_{n\leq N}\frac{\Lambda\left(n\right)}{n^{2}}
\]
uniformly in $M$. Obviously we will have also the ``Von Mangoldt
terms'' 
\[
\frac{1}{2}\sum_{n\leq N}\Lambda\left(n\right)\Lambda\left(n+M\right)+\frac{1}{2}\sum_{n\leq N}\Lambda\left(n\right)^{2}=\frac{r_{PT}\left(N,M\right)}{2}+\frac{r_{PT}\left(N,0\right)}{2}
\]

\subsection{Put together all the pieces}

Expanding $\psi^{2}\left(N\right)$ in \eqref{eq:prime in tuples id base}
and observing that some terms cancel each other out (for example the
term $N^{2}$ in Section $5.1$ or the double series $\sum_{\rho_{1}}\frac{N^{\rho_{1}}}{\rho_{1}}\sum_{\rho_{2}}\frac{N^{\rho_{2}}}{\rho_{2}}$
in Section $5.2$) we finally get
\begin{align*}
\sideset{}{^{\prime}}\sum_{h=0}^{M}r_{PT}\left(N,h\right)= & NM+\sum_{\rho}\frac{N^{\rho+1}}{\rho\left(\rho+1\right)}-\sum_{\rho}\frac{\left(N+M\right)^{\rho+1}}{\rho\left(\rho+1\right)}+\sum_{\rho}\frac{\left(2+M\right)^{\rho+1}}{\rho\left(\rho+1\right)}\\
- & \left(\sum_{\rho_{1}}\sum_{\rho_{2}}\frac{M^{\rho_{1}+\rho_{2}}\left(-1\right)^{\rho_{2}+1}}{\rho_{2}}\left(B_{-N/M}\left(\rho_{2}+1,\rho_{1}\right)-B_{-2/M}\left(\rho_{2}+1,\rho_{1}\right)\right)\right)\cdot1\left(M\right)\\
- & \left(\sum_{\rho_{1}}\sum_{\rho_{2}}\frac{N^{\rho_{1}+\rho_{2}}-2^{\rho_{1}+\rho_{2}}}{\rho_{2}\left(\rho_{1}+\rho_{2}\right)}\right)\cdot\widetilde{1}\left(M\right)+\sum_{\rho_{1}}\frac{N^{\rho_{1}}}{\rho_{1}}\sum_{\rho_{2}}\frac{\left(N+M\right)^{\rho_{2}}}{\rho_{2}}\\
+ & \frac{\Lambda\left(N\right)}{2}\sum_{\rho}\frac{N^{\rho}}{\rho}-\frac{\Lambda\left(N\right)}{2}\sum_{\rho}\frac{\left(N+M\right)^{\rho}}{\rho}\\
+ & G\left(N,M\right)
\end{align*}
where 
\[
G\left(N,M\right)\ll\begin{cases}
N\left(M+1\right)\exp\left(-C\sqrt{\log\left(N\right)}\right), & \mathrm{without}\,\mathrm{RH}\\
\sqrt{N}\left(M+1\right)\log^{2}\left(N\right), & \mathrm{with}\,\mathrm{RH}
\end{cases}
\]
as claimed.

It is reasonable to think that a truncated version of this formula
can be done with the same strategy we used in Sections $3.7-3.13$;
it will be the subject of future research.

email address: cantarini\_m@libero.it
\end{document}